\documentclass[a4paper,11pt]{amsart}

\usepackage[top=3.3cm, bottom=3cm, left=2.5cm, right=2.5cm]{geometry}
\usepackage{latexsym,amsmath,amsthm,verbatim,ifthen,amssymb,mathtools,booktabs}

\usepackage{hyperref}
\usepackage{caption}
\usepackage{microtype}
\usepackage{graphicx, adjustbox, multirow}

\usepackage{enumerate}
\usepackage{palatino}
\usepackage{mathpazo}

\usepackage[all]{xy}
\usepackage{tikz}
\usetikzlibrary{positioning,arrows,cd}

\newcommand{\id}{{\rm Id}}
\newcommand{\tc}{{\rm TC}}
\newcommand{\cd}{{\rm cd}}
\newcommand{\cat}{{\rm cat}}
\newcommand{\secat}{{\rm secat}}

\newcommand{\height}{{\rm height}}
\newcommand{\fH}{\langle H\rangle}
\newcommand{\mZ}{\mathbb{Z}}

\DeclareMathOperator{\Hom}{Hom}
\DeclareMathOperator{\Ext}{Ext}

\newcommand{\E}{{\rm E}}

\newtheorem{theorem}{Theorem}[section]

\newtheorem{proposition}[theorem]{Proposition}
\newtheorem{lemma}[theorem]{Lemma}
\newtheorem{corollary}[theorem]{Corollary}

\theoremstyle{definition}
\newtheorem{definition}[theorem]{Definition}

\theoremstyle{remark}
\newtheorem{example}[theorem]{Example}					
\newtheorem{remark}[theorem]{Remark}

\begin{document}
	
	\title[On the sectional category of subgroup inclusions]{On the sectional category of subgroup inclusions and Adamson cohomology theory}
	
	\author[Z. B\l{}aszczyk]{Zbigniew B\l{}aszczyk}
	\address{Faculty of Mathematics and Computer Science,
		Adam Mickiewicz University, Umultowska 87, 60-479 Pozna\'n, Poland.}
	\email{blaszczyk@amu.edu.pl}
	
	\author[J. Carrasquel-Vera]{Jos\'e Gabriel Carrasquel-Vera}
	\address{Faculty of Mathematics and Computer Science,
		Adam Mickiewicz University, Umultowska 87, 60-479 Pozna\'n, Poland.}
	\email{jgcarras@amu.edu.pl, jgcarras@gmail.com}
	
	\author[A. Espinosa Baro]{Arturo Espinosa Baro}
	\address{Faculty of Mathematics and Computer Science,
		Adam Mickiewicz University, Umultowska 87, 60-479 Pozna\'n, Poland.}
	\email{arturo.espinosabaro@gmail.com, artesp1@amu.edu.pl}
	
	
	\begin{abstract}
		The sectional category of a subgroup inclusion $H \hookrightarrow G$ can be defined as the sectional category of the corresponding map between Eilenberg--MacLane spaces. We extend a characterization of topological complexity of aspherical spaces given by Farber, Grant, Lupton and Oprea to the context of sectional category of subgroup inclusions and investigate it by means of Adamson cohomology theory. 
	\end{abstract}
	
	\maketitle
	
\makeatletter \def\l@subsection{\@tocline{2}{0pt}{2pc}{6pc}{}}  \makeatother
	
	
\section*{Introduction}
	The \textit{sectional category} of a map $f \colon X \to Y$, written $\secat(f)$, is defined to be the smallest integer $n \geq 0$ such that there exists an open cover $U_0$, \ldots, $U_n$ of~$Y$ and continuous maps $s_i \colon U_i \to X$ with the property that $f \circ s_i$ is homotopic to the inclusion $U_i \hookrightarrow Y$ for any $0 \leq i \leq n$. This concept was introduced by Schwarz~\cite{Schwarz66} under the name \textit{genus} for $f$ a fibration, and it was subsequently generalized to arbitrary maps by Fet \cite{Fet} and Berstein--Ganea~\cite{BernsteinGanea}.
	
	A particular case of sectional category is the famous Lusternik--Schnirelmann category. Another interesting instance is the \textit{topological complexity} of a space~$X$, denoted $\tc(X)$, which arises as the sectional category of the path space fibration
	\[ X^{[0,1]} \to X \times X, \; \gamma \mapsto \big(\gamma(0), \gamma(1)\big). \]
	It was introduced by Farber \cite{Farber03, Farber08} in order to tackle the motion planning problem in robotics from a topological perspective and quickly proved to be an interesting invariant in its own right. 
	
	A problem that has fuelled much of $\tc$-related work is the characterization of topological complexity of Eilenberg--MacLane $K(\pi,1)$ spaces. Given that topological complexity is invariant under homotopy equivalence and that the homotopy type of a $K(\pi,1)$ space is determined by $\pi$, it is common to define $\tc(\pi)$ as $\tc\big(K(\pi,1)\big)$. Farber \cite{Farber06b} asked whether one can express $\tc(\pi)$ in terms of algebraic properties of $\pi$. This question is motivated by a phenomenon that occurs for Lusternik--Schnirelmann category of $K(\pi,1)$ spaces, recorded in a celebrated paper of Eilenberg and Ganea~\cite{EG65}. Leaving certain fringe cases aside, their result states that $\cat\big(K(\pi,1)\big) = \cd\,\pi$, where $\cd\,\pi$ stands for the \textit{cohomological dimension} of $\pi$, i.e. the length of the shortest possible projective $\mathbb{Z}[\pi]$-resolution of the trivial $\mathbb{Z}[\pi]$-module $\mathbb{Z}$ or, equivalently, the greatest integer $n \geq 0$ such that $H^n(\pi, M) \neq 0$ for some $\mathbb{Z}[\pi]$-module~$M$. 
	
	Up until recently, any progress in this context was mostly related to a specific choice of a family of groups: choose a family of groups, then use its characteristic features (e.g. a particularly well understood cohomology ring or a specific subgroup structure) to deduce, or at least estimate, topological complexity of its members. Perhaps the most comprehensive result in this direction was obtained by Farber and Mescher~\cite{FarberMes}. They proved that if a group $\pi$ is hyperbolic in the sense of Gromov and it admits a compact model of a $K(\pi,1)$ space, then its topological complexity is equal to either $\cd\,(\pi\times \pi)$ or $\cd\,(\pi\times \pi)-1$. Continuing on that line of thought, Dranishnikov in \cite{Dranish20} improved that estimation by showing that, when $\pi$ is hyperbolic, $\tc(\pi) = \cd(\pi \times \pi) = 2 \cd(\pi)$. Further on, very recently Li \cite{Li20} showed a generalization of this for certain toral relatively hyperbolic groups. However, in a recent breakthrough, Farber, Grant, Lupton and Oprea \cite{FGLO17} related $\tc(\pi)$ to invariants coming from equivariant Bredon cohomology. More specifically, they proved that 
	\[ \tc(\pi) \leq \cd_{\langle\Delta_\pi\rangle} (\pi\times \pi),\]
	where $\cd_{\langle\Delta_\pi\rangle} (\pi\times \pi)$ denotes the cohomological dimension of $\pi\times \pi$ with respect to the family of subgroups of~$\pi \times \pi$ ``generated by'' the diagonal subgroup $\Delta_\pi$. This number can be seen as the smallest possible dimension of $E_{\langle\Delta_{\pi}\rangle}(\pi \times \pi)$, the classifying space for $(\pi \times\pi)$-actions with isotropy groups in the family $\langle\Delta_{\pi}\rangle$. (See Subsections \ref{sect:secat_characterization} and \ref{sect:AdamsonvsBredon} for details.)
	
	The objective of this article is twofold. One, we begin a systematic study of the sectional category of subgroup inclusions: given a group $G$ and its subgroup~$H$, we define $\secat(H \hookrightarrow G)$ as the sectional category of the corresponding map between Eilenberg--MacLane spaces. This setting includes $\tc$, as the topological complexity of $X$ can be seen as the sectional category of the diagonal inclusion $X \to X \times X$, so that $\tc(\pi) = \secat(\Delta_{\pi} \hookrightarrow \pi \times \pi)$. In fact, the cornerstone of \cite{FGLO17}, a characterization of $\tc(\pi)$ as the smallest integer $n \geq 0$ such that a certain canonical $(\pi \times \pi)$-equivariant map $E(\pi \times \pi) \to E_{\langle \Delta \pi\rangle}(\pi\times \pi)$ can be equivariantly deformed into the $n$-dimensional skeleton of $E_{\langle \Delta_\pi\rangle}(\pi\times\pi)$, has a generalization to this more general context. We also describe and develop a ``relative canonical class'' analogous to the one developed by Berstein and Schwarz for the study of Lusternik--Schnirelmann category theory. Two, we introduce the Adamson cohomology theory \cite{Adamson} into the study of $\secat(H \hookrightarrow G)$, hence also into the study of $\tc(\pi)$. In particular, we exhibit a relationship between the ``zero-divisors'' of $H^*(G, M) \to H^*(H,M)$, which provide a lower bound for $\secat(H\hookrightarrow G)$, and the Adamson cohomology of the pair $(G,H)$.\medskip

	\noindent\textbf{Notation.} Throughout the paper $G$ is a discrete group and $H \subseteq G$ its fixed subgroup, and $G/H$ denotes the set of left cosets of $H$ in $G$ equipped with a canonical $G$-action. Whenever we specialize to the setting of topological complexity, we take $G = \pi \times \pi$ and $H = \Delta_{\pi}$, the diagonal subgroup of $\pi \times \pi$. 
	
	Given a short exact sequence of $G$-modules
	\[0 \rightarrow A\stackrel{i}{\rightarrow} B\rightarrow C\rightarrow 0,\] 
	and a $G$-module map $f\colon B\to M$ with $f\circ i=0$, we will write $\hat{f}$ for the induced map
	
	\[\begin{tikzcd}
		A\ar[r, "i"]&B\ar[d, "f"]\ar[r]&C\ar[dl, "\hat{f}"] \\
		&M.
	\end{tikzcd}\]
	
	\section{Sectional category of subgroup inclusions}
	
	The inclusion $i \colon H \hookrightarrow G$ induces a map $K(i,1) \colon K(H,1) \to K(G,1)$ between the corresponding Eilenberg--MacLane spaces which satisfies $\pi_1\big(K(i,1)\big)=i$. Define the \emph{sectional category} of the inclusion $H \hookrightarrow G$, denoted $\secat(H\hookrightarrow G)$, as the sectional category of $K(i,1) \colon K(H,1) \to K(G,1)$. Due to homotopy invariance of sectional category, $\secat(H \hookrightarrow G)$ depends only on the conjugacy class of $H$ in $G$.
	
	Further on we will make use of some basic properties of sectional category, so let us recall them. For a detailed proof, we refer the interested reader to \cite{Schwarz66}.
	
	\begin{proposition}\label{propertysecat}
		Let $F \rightarrow E \xrightarrow{p} B$ a fibration, then
		
		\begin{enumerate}
			\item $\secat(p) \leq \cat(B)$.
			\item Let $k > 0$ the maximal integer such that there exist $$ u_1, \cdots, u_k \in \ker \{ \tilde{H}^*(B,R) \xrightarrow{p^*} \tilde{H}^*(E,R) \} $$ with $u_1 \smile \cdots \smile u_k \neq 0$. Then $\secat(p) \geq k$.
			\item $\secat(p) \leq k$ if and only if the $(k+1)$-fold fibrewise join of $p$ (see below) has a section.  
		\end{enumerate} 
		
	\end{proposition}
	
	\subsection{A characterization of $\secat(H \hookrightarrow G)$}\label{sect:secat_characterization} 
	
	Recall that a family $\mathcal{F}$ of subgroups of $G$ is said to be \textit{full} provided that it is non-empty and is closed under conjugation and the taking of subgroups. This also entails the fact that the family is closed under taking intersections as well. We will write $\langle H \rangle$ for the smallest full family of subgroups of $G$ containing $H$. The \textit{classifying space of $G$ with respect to $\mathcal{F}$} is a $G$-CW complex $E_{\mathcal{F}}G$ satisfying the following conditions:
	
	\begin{itemize}
		\item every isotropy group of $E_{\mathcal{F}}G$ belongs to $\mathcal{F}$,
		\item for any $G$-CW complex $X$ with all isotropy groups in $\mathcal{F}$ there exists a unique (up to $G$-equivariant homotopy) $G$-equivariant map $X \to E_{\mathcal{F}}G$. 
	\end{itemize}
	In particular, there is a unique $G$-equivariant map $EG \to E_{\mathcal{F}}G$, where $EG$ is the classifying space of $G$ with respect to the family consisting of the trivial subgroup, or, in other words, the universal cover of a $K(G,1)$ space. Further properties of $E_{\mathcal{F}}G$ are discussed at length in \cite{Luck}.
	
	The aim of this subsection is to prove the following result.
	
	\begin{theorem}\label{thm:secat_FGLO_char}
		The sectional category of $H \hookrightarrow G$ coincides with the minimal integer $n\geq 0$ such that the $G$-equivariant map $\rho \colon EG \to E_{\langle H\rangle}G$ can be $G$-equivariantly factored up to $G$-homotopy as
		$$\begin{tikzcd} EG \arrow[rr, "\rho"] \arrow[rd] & & E_{\langle H \rangle} G \\
			& (E_{\langle H \rangle} G)_n, \arrow[ru] \end{tikzcd} $$ 
		where $(E_{\langle H \rangle} G)_n$ denotes the $n$-skeleton of $E_{\langle H \rangle} G$.
	\end{theorem}
	
	\noindent This is a generalization of Farber, Grant, Lupton and Oprea's \cite[Theorem~3.3]{FGLO17}, where $\tc(\pi)$ is described as the minimal integer $n\geq 0$ such that the $(\pi \times \pi)$-equivariant map $E(\pi \times \pi) \to E_{\langle\Delta_{\pi}\rangle}(\pi \times \pi)$ can be equivariantly deformed into the $n$-skeleton of $E_{\langle\Delta_{\pi}\rangle}(\pi \times \pi)$. Our approach closely follows theirs. This was very recently proved independently in \cite{CLM}, albeit by different means. The next lemma is an abstraction of an intermediate step in the proof of \cite[Theorem 2.1]{FGLO17}. 
	
	\begin{lemma}\label{lemma:secat_FGLO_char_aux}
		We have $\secat(H \hookrightarrow G) \leq n$ if and only if the Borel fibration
		\[ p_{n}\colon EG \times_G \ast^{n+1} (G/H) \rightarrow EG/G \] 
		has a section, where $\ast^{n+1} (G/H)$ denotes the $(n+1)$-fold join of $G/H$.
	\end{lemma}

	Recall that the \textit{fibrewise join} of a fibration $F \to E \stackrel{p}{\rightarrow} B$ is a fibration $p * p \colon E *_B E \to B$ whose fibre has the homotopy type of the join $F *F$ and whose total space is given by
	\[ \big\{(x,y,t) \in E \times E \times [0,1] \,\big|\, p(x) = p(y) \big\} \]
	modulo the relations $(x,y,0) \sim (x',y,0)$ and $(x,y,1) \sim (x,y',1)$. It is convenient to think of the elements of the total space of the $n$-fold fibrewise join of $p$ as formal sums of the form $\sum_{i=1}^n t_i x_i$, where every $x_i$ is understood to lie in the same fibre of $p$, and $t_i$'s are non-negative real numbers such that $\sum_{i=1}^n t_i = 1$.
	
	\begin{proof}
		The map $EG \times_G (G/H) \to EG/H$ given by $G(x, gH) \mapsto Hg^{-1}x$ is easily seen to be a homeomorphism which commutes with projections onto $EG/G$. Consequently, $p_0$ is isomorphic to the fibration $EG/H \to EG/G$, which is a model for the map $K(H,1) \to K(G,1)$.
		
		It follows that $\secat(H \hookrightarrow G) = \secat(p_0)$. By \ref{propertysecat}, $\secat(p_0)\leq n$ if and only if the $(n+1)$-fold fibrewise join of $p_0$ has a section. Thus in order to conclude the proof, it remains to verify that the $(n+1)$-fold fibrewise join of $p_0$ coincides with~$p_n$. To this end, note that the map $EG \times_G \ast^{n+1} (G/H) \to \ast^{n+1}_{EG/G}\big(EG \times_G (G/H)\big)$ given by
		\[ G\left(x, \sum_{i=1}^{n+1} t_i g_iH\right) \mapsto \sum_{i=1}^{n+1} t_iG(x, g_iH) \]
		is a homeomorphism which commutes with projections onto $EG/G$. 
	\end{proof}
	
	\begin{proof}[Proof of Theorem \ref{thm:secat_FGLO_char}] In view of \cite[{Chapter 4, Theorem 8.1}]{Husemoller66}, sections of the fibration \[p_n \colon EG \times_G \ast^{n+1} (G/H) \rightarrow EG/G\] introduced in Lemma \ref{lemma:secat_FGLO_char_aux} are in one-to-one correspondence with $G$-equivariant maps $EG \to *^{n+1} (G/H)$. Consequently, Lemma \ref{lemma:secat_FGLO_char_aux} can be restated as saying that $\secat(H \hookrightarrow G)$ coincides with the minimal integer $n \geq 0$ such that there exists a $G$-equivariant map $EG \to *^{n+1} (G/H)$. 
		
		Let $m\geq 0$ be the minimal integer such that the $G$-equivariant map $EG \to E_{\langle H\rangle}G$ can be deformed into the $m$-dimensional skeleton of $E_{\langle H\rangle}G$. We will now use the fact that the infinite join $*^{\infty} (G/H)$ is a model for $E_{\langle H \rangle}G$, as explained in~\cite[Section IV]{Blowers}. Given that $\dim *^{n+1} (G/H) = n$, the existence of a $G$-equivariant map $EG \to *^{n+1} (G/H)$ implies the existence of a $G$-equivariant map $EG \to *^{n+1} (G/H) \to (E_{\fH} G)_n \to E_{\fH}G$ by the equivariant cellular approximation theorem. Since any two $G$-equivariant maps $EG \to E_{\langle H \rangle}G$ are $G$-equivariantly homotopic, this last composition is $\rho$ and we see that $\secat(H \hookrightarrow G) \geq m$.
		
		On the other hand, the $G$-equivariant map $(E_{\langle H\rangle}G)_m \to *^{\infty}(G/H)$ yields a $G$-equivariant map $(E_{\langle H\rangle}G)_m \to *^{m+1}(G/H)$ by the equivariant Whitehead theorem \cite[{Chapter 1, Theorem 3.1}]{May}. This in turn implies the existence of a $G$-equivariant map $EG \to (E_{\langle H\rangle}G)_m \to *^{m+1}(G/H)$, which shows that $\secat(H\hookrightarrow G) \leq m$. 
	\end{proof}
	
	As an immediate corollary, we obtain a generalization of [\cite{FGLO17}, Corollary 3.5.1]:
	
	\begin{corollary}\label{secatlessdim}
		Let $H \hookrightarrow G$ be a monomorphism of groups, then \[ \secat(H \hookrightarrow G) \leq \dim E_{\langle H \rangle}G. \] 
	\end{corollary} 
	
	\subsection{The Berstein class of $G$ relative to $H$}\label{sect:Bernstein_class}
	
	We will now recall a construction from \cite{DranishRudyak09}. Take the usual augmentation associated to a group $G$ $$0\hookrightarrow K \hookrightarrow \mathbb{Z}[G] \xrightarrow{\varepsilon} \mathbb{Z} \rightarrow 0 $$ and $K^{\otimes n}$ with $G$-module structure induced by the diagonal action of $G$ on the tensor product of copies of $K$. Given that $K^{\otimes n}$ is a free abelian group, and that there exists an isomorphism $g \otimes m \mapsto g \otimes g m$ from $\mathbb{Z}[G] \otimes K^{\otimes n}$ with action on the first factor to $\mathbb{Z}[G] \otimes K^{\otimes n}$ with the diagonal one, we have that $\mathbb{Z}[G] \otimes K^{\otimes n}$ is a free $G$-module. Splicing together short exact sequences of $G$-modules
	$$ 0 \rightarrow K^{\otimes n+1} \rightarrow \mathbb{Z}[G] \otimes K^{\otimes n} \xrightarrow{\varepsilon \otimes \textrm{id}} K^{\otimes n} \rightarrow 0 $$ 
	yields a $G$-module free resolution of $\mathbb{Z}$ 
	$$ \cdots \rightarrow \mathbb{Z}[G]\otimes K^{\otimes n+1} \rightarrow \mathbb{Z}[G]\otimes K^{\otimes n} \rightarrow \cdots \rightarrow \mathbb{Z}[G] \xrightarrow{\varepsilon} \mathbb{Z} \to 0 $$ 
	that will be denoted by $\mathcal{G}$.
	
	We now give a simple description of the cup product using $\mathcal{G}$.
	
	\begin{proposition}\label{pro:DescriptCup}
		Let $[a]\in H^p(G,A)$ and $[b]\in H^q(G,B)$ be cohomology classes represented by cocycles $a\colon \mathbb{Z}[G]\otimes K^{\otimes p}\to A$ and $b\colon\mathbb{Z}[G]\otimes K^{\otimes q}\to B$. Then the cup product $[a][b]\in H^{p+q}(G,A\otimes B)$ is represented by the map \[\mathbb{Z}[G]\otimes K^{\otimes p+q}\xrightarrow{\varepsilon\otimes\id}K^{\otimes p+q}\xrightarrow{\hat{a}\otimes\hat{b}} A\otimes B.\]
	\end{proposition}
	
	\begin{proof}
		Denote by $\mathcal{F}$ the standard resolution of $\mathbb{Z}$ as a $G$-module and consider a map $\varphi\colon \mathcal{F}\to\mathcal{G}$ defined by $\varphi_p\colon\mZ[G^{p+1}]\to \mZ[G]\otimes K^{\otimes p}$, \[\varphi_p(x_0,x_1\ldots,x_p)=x_0\otimes(x_1-x_0)\otimes\cdots\otimes(x_p-x_{p-1}).\] The combinatorics in the proof of \cite[Lemma 3.1]{FarberMes} show that this is in fact a chain map. The result follows from the commutativity of the following $G$-module diagram
		\[\begin{tikzcd}
			\mathcal{F}_{p+q}\ar[r,"\Delta"]\ar[d,"\varphi"']&(\mathcal{F}\otimes\mathcal{F})_{p+q}\ar[r,"\varphi\otimes\varphi"]&(\mathcal{G}\otimes\mathcal{G})_{p+q}\ar[d, "a\otimes b"]\\
			\mathcal{G}_{p+q}\ar[r,"\varepsilon\otimes\id"']&K^{\otimes p+q}\ar[r,"\hat{a}\otimes\hat{b}"']&A\otimes B,
		\end{tikzcd}\]
		where $\Delta$ denotes the Alexander--Whitney diagonal map, see \cite[p. 108]{Brown82}, and the action on tensor products is diagonal.
	\end{proof}
	
	Consider a permutation $G$-module $\mathbb{Z}[G/H]$ and write $I$ for the kernel of the augmentation homomorphism $\mathbb{Z}[G/H] \rightarrow \mathbb{Z}$, given by $gH \mapsto 1$ for any $gH \in G/H$. Define a $G$-module homomorphism $ \xi \colon \mathbb{Z}[G]\otimes K \rightarrow I $ as the composition of $\varepsilon \otimes \textrm{id}$ and the map $\mu\colon K \rightarrow I$ induced by the canonical projection $G \rightarrow G/H$ in 
	\[ \begin{tikzcd} 
		K\ar[r]\ar[d]&\mathbb{Z}[G]\ar[r, "\varepsilon"] \ar[d]&\mathbb{Z}\ar[d, "="]\\
		I\ar[r]&\mathbb{Z}[G/H]\ar[r]&\mathbb{Z}.
	\end{tikzcd} \] This is obviously a cocycle, and thus it represents a one-dimensional cohomology class $\omega \in H^1(G,I)$, which will be called the \emph{Berstein class of $G$ relative to $H$}. By Proposition \ref{pro:DescriptCup}, its $n$-th power $\omega^n\in H^n(G,I^{\otimes n})$ is represented by the map \[\mZ[G]\otimes K^{\otimes n}\xrightarrow{\varepsilon\otimes\id} K^{\otimes n}\xrightarrow{\mu^{\otimes n}} I^{\otimes n}.\]
	
	\begin{lemma}
		The class $\omega$ defined as above is a zero-divisor, i.e. $$\omega \in \ker\!\big[H^1(G, I) \to H^1(H,I)\big].$$
	\end{lemma}
	\begin{proof}
		The $H$-module homomorphism $\mathbb{Z}[G] \to I$ defined by $g\mapsto gH-H$ shows that $\xi$ considered as an $H$-module homomorphism is a coboundary.
	\end{proof}

	In the particular case of $G = \pi \times \pi$ and $H = \Delta_{\pi}$ is the diagonal subgroup of $\pi \times \pi$, this class coincides with the ``canonical $\tc$-class'' introduced by Costa and Farber \cite{Costa}. This follows from the fact that $(\pi \times \pi)/\Delta_{\pi}$ and $\pi$ seen as a $(\pi \times \pi)$-set via the action $(g,h)x = gxh^{-1}$ are isomorphic as $(\pi\times \pi)$-sets.\\

	Berstein in \cite{Bernstein} showed that for a connected $CW$-complex $X$ of dimension $n \geq 3$, $\cat(X) = n$ if and only if there exists a class $u \in H^1(X,K)$ such that $u^n \neq 0$ in $H^n(X,K^{\otimes n})$, where $K$ is the augmentation ideal of $\mathbb{Z}[\pi_1(X)]$. Costa and Farber gave a version of this result for $\tc$ (\cite[Theorem 7]{Costa}). Here, we state an analogue of their result illustrating the relation between $\secat(H\hookrightarrow G)$ and the Berstein canonical class of $G$ relative to $H$.  
	
	\begin{theorem}\label{th:GenerCostaFarber}
		If $n = \cd\,G \geq 3$, then $\secat(H \hookrightarrow G) \leq n-1$ if and only if $\omega^n = 0$.
	\end{theorem}
	
	We postpone the proof to the end of this article.
	
	\section{Adamson cohomology and sectional category}
	
	In this section we  briefly review Adamson cohomology, a theory first introduced by Adamson \cite{Adamson} for finite groups. Later Hochschild \cite{Hoch} generalized the ideas of Adamson to develop a homological algebra theory in the relative setting. Then we proceed to recast Adamson cohomology in terms of equivariant Bredon cohomology. 
	
	\subsection{Review of the Adamson cohomology theory}
	
	Recall that an exact sequence of $G$-modules
	$$ \cdots \rightarrow M_{i} \xrightarrow{f_{i +1}} M_{i} \xrightarrow{f_{i}} M_{i-1} \rightarrow \cdots \ $$ 
	is said to be \textit{$(G,H)$-exact} provided that $\ker f_i$ is a direct summand of $M_i$ as an $H$-module for each $i$. A $G$-module P is said to be \textit{$(G,H)$-projective} provided that for every short $(G,H)$-exact sequence of modules $M \xrightarrow{f} N \rightarrow 0$ and every $G$-homomophism $g\colon P \rightarrow N$, there exists a $G$-homomorphism $h\colon P \rightarrow M$ making the diagram 
	$$\begin{tikzcd} & P \arrow[ld, "h"', dashrightarrow] \arrow[d, "g"]\\
		M \arrow[r, "f"] & N \arrow[r] & 0
	\end{tikzcd}$$
	commutative. Finally, given a $G$-module $M$, a \textit{$(G,H)$-projective resolution} of $M$ is an $(G,H)$-exact sequence of $G$-modules
	$$ \cdots \rightarrow P_n \rightarrow P_{n-1} \rightarrow \cdots \rightarrow P_0 \rightarrow M \rightarrow 0 $$ 
	such that $P_i$ is $(G,H)$-projective for each $i \geq 0$.  
	
	\begin{example}
		Given $n \geq 0$, define $C_n(G/H)$ to be the permutation module $\mathbb{Z}\big[(G/H)^{n+1}\big]$, where $(G/H)^{n+1}$ is equipped with the diagonal $G$-action, i.e 
		$$ g(g_0H, \ldots, g_nH) = (gg_0H, \ldots, gg_nH)$$ 
		
		Furthermore, let $d_n \colon C_n(G/H) \to C_{n-1}(G/H)$ be given by 
		$$d_n(g_0H, \ldots, g_nH) = \sum_{i=0}^n (-1)^i (g_0H, \ldots, \widehat{g_nH}, \ldots, g_nH),$$ 
		where $\widehat{g_nH}$ means that the element $g_nH$ is removed from the tuple. 
		Hochschild \cite{Hoch} proved that $(C, d)$ forms a $(G,H)$-projective resolution of the trivial $G$-module $\mathbb{Z}$, with the augmentation map defined by sending every coset to $1$. This resolution will be called the \emph{standard resolution of $G$ relative to $H$}.
	\end{example}
	
	In an analogous way to the non-relative case, Hochschild defined the relative extension functor as 
	$$ \Ext^n_{(G,H)}(M,N) = H^n\big(\Hom_{G}(\mathcal{P}_*,N)\big), $$ 
	where $M$ and $N$ are $G$-modules, and $\mathcal{P}_*$ is a $(G,H)$-projective resolution of $M$. Then the \textit{Adamson cohomology} of $G$ with respect to $H$ with coefficients in a $G$-module $M$ is 
	$$H^*\big([G:H],N\big) = \Ext^{*}_{(G,H)}(\mathbb{Z},N), $$ 
	where $\mathbb{Z}$ is the trivial $G$-module. The \emph{Adamson cohomological dimension} of $G$ relative to $H$, defined as the length of the shortest possible $(G,H)$-projective resolution of $\mathbb{Z}$, will be denoted by $\cd\,[G:H]$. This number can be equivalently characterized as the maximal integer $n \geq 0$ such that $H^n\big([G:H], M\big) \neq 0$ for some $G$-module $M$, as in spirit of \cite[Chapter 4, Lemma 4.1.6]{Weibel}.
	
	Note that the Adamson relative cohomology defined like this can be seen as a particular case of the cohomology of a permutation representation, with $G/H$ as the base $G$-set, see Blowers \cite{Blowers}.
	
	\subsection{Adamson canonical class and its universality}

	Consider the short exact sequence of $G$-modules,
	$$ 0 \rightarrow I \hookrightarrow \mathbb{Z}[G/H] \xrightarrow{\varepsilon} \mathbb{Z} \rightarrow 0. $$ 
	Tensoring it over $\mathbb{Z}$ with $I^{\otimes k-1}$, the $k-1$-fold tensor power of $I$ over $\mathbb{Z}$ seen as a $G$-module via the diagonal $G$-action, yields another short exact sequence:
	$$ 0 \rightarrow I^{\otimes k} \hookrightarrow \mathbb{Z}[G/H]\otimes I^{\otimes k-1} \xrightarrow{\varepsilon \otimes \textrm{id}} I^{\otimes k-1} \rightarrow  0. $$ 
	Splicing all those sequences together for varying $k$ yields an exact sequence
	$$ \cdots \rightarrow \mathbb{Z}[G/H] \otimes I^{\otimes k} \rightarrow \mathbb{Z}[G/H] \otimes I^{\otimes k-1} \rightarrow \cdots \rightarrow \mathbb{Z}[G/H] \rightarrow \mathbb{Z} \rightarrow 0. $$
	This is a $(G,H)$-projective resolution. To see $(G,H)$-exactness, note the decomposition as an $H$-module
	$$ \mathbb{Z}[G/H] \otimes I^{\otimes k} \cong (\mathbb{Z} \otimes I^{\otimes k-1}) \oplus (I \otimes I^{\otimes k-1}). $$ 
	In order to see projectiveness, define the inverse maps 
	$$\alpha\colon \mathbb{Z}[G/H] \otimes I  \rightarrow \mathbb{Z}[G] \otimes_H I, \textnormal{ $\alpha(\overline{x} \otimes y) = (x \otimes x^{-1}y) $} $$ and 
	$$ \beta\colon \mathbb{Z}[G] \otimes_H I \rightarrow \mathbb{Z}[G/H] \otimes I, \textnormal{ $\beta(x\otimes y) = \overline{x} \otimes xy$}.$$
	We define the $G$-action on $\mathbb{Z}[G] \otimes_H I$ such that is compatible with the diagonal one in $\mathbb{Z}[G/H] \otimes I$. As such, define the $G$-module structure on $\mathbb{Z}[G] \otimes_H I$ by the action $$g (x \otimes y) = \alpha(g (\overline{x} \otimes xy)).$$ As a consequence, we have
	$$g (x \otimes y) = \alpha(g (\overline{x} \otimes xy)) = \alpha(\overline{gx} \otimes gxy) = gx \otimes y, $$
	and we see that the action restricts to the first component. Then we use \cite[Lemma 2]{Hoch} and thus, generalizing this morphism to $\mathbb{Z}[G/H] \otimes I^{\otimes n}$ for every $n > 0$, we have that every term in the exact sequence constructed above is $(G,H)$-projective.
	
	The $(G,H)$-projective resolution above lets us define a cup product on Adamson cohomology as in Proposition \ref{pro:DescriptCup}.
	
	\begin{definition}
		Let $[a]\in H^p([G:H],A)$ and $[b]\in H^q([G:H],B)$ be cohomology classes represented by cocycles $a\colon \mathbb{Z}[G/H]\otimes I^{\otimes p}\to A$ and $b\colon\mathbb{Z}[G/H]\otimes I^{\otimes q}\to B$. Define the \emph{cup product} $[a][b]\in H^{p+q}([G:H],A\otimes B)$ as the class represented by the map \[\mathbb{Z}[G/H]\otimes I^{\otimes p+q}\xrightarrow{\varepsilon\otimes\id} I^{\otimes p+q}\xrightarrow{\hat{a}\otimes\hat{b}} A\otimes B.\]
	\end{definition}
	
	It is easy to check that this product verifies the properties \emph{dimension} $0$, \emph{naturality with respect to coefficient homomorphisms}, \emph{compatibility with} $\delta$, \emph{associativity} and \emph{commutativity} analogous to the ones in \cite[Pg. 110]{Brown82}.

	\begin{definition}
		The \emph{Adamson canonical class} $\phi \in H^1\big([G:H], I\big)$ is the class represented by the cocycle $\mathbb{Z}[G/H] \otimes I \xrightarrow{\varepsilon \otimes id} I$. Also, $\height(\phi)$ is the largest $n \geq 0$ such that $$\phi^n \in H^n([G:H],I^{\otimes n})$$ is nonzero.
	\end{definition}
	
	The Adamson canonical class is universal in the following sense.
	
	\begin{proposition}\label{prop:universality}
		For any $G$-module $M$ and any class $\lambda \in \text{H}^n\big([G:H],M\big)$ there exists a $G$-homomorphism $h\colon I^{\otimes n} \rightarrow M$ such that $h^*(\phi^n) = \lambda$.
	\end{proposition}
	
	\begin{proof}
		Let $f\colon \mZ[G/H]\otimes I^{\otimes n}\rightarrow M$ be a cocycle representing the class $\lambda \in H^n([G:H],M)$. By the definition of the cup product in Adamson cohomology, the class $\phi^n$ is represented by \[\mZ[G/H]\otimes I^{\otimes n}\xrightarrow{\varepsilon\otimes \id}I^{\otimes n}.\]Taking $h=\hat{f}$ we see that $h^*(\phi^n)=\lambda$.
	\end{proof}

	\begin{corollary}
		Let $\phi \in H^1([G:H],I)$ be the Adamson canonical class, we have that
		$$ \cd\,[G:H] = \height(\phi).$$
	\end{corollary}
	
	It is possible to characterize Adamson cohomology groups in terms of zero divisors for certain coefficient systems.
	
	\begin{proposition}\label{prop:Adamson_zerodiv}
		For any $G$-module $M$ and $n\ge 1$, we have
		\[H^n([G:H],M)=\ker \left[H^1(G,\Hom_{\mathbb{Z}}(I^{\otimes n-1},M))\to H^1(H,\Hom_{\mathbb{Z}}(I^{\otimes n-1},M))\right].\]
		In particular,
		\[H^1([G:H],M)=\ker \left[H^1(G,M)\to H^1(H,M)\right].\]
	\end{proposition}
	
	\begin{proof}
		Consider the short exact sequence $$ 0 \rightarrow I^{\otimes n} \rightarrow \mathbb{Z}[G/H] \otimes I^{\otimes n-1} \rightarrow I^{\otimes n-1} \rightarrow 0. $$ Applying the Ext functor we obtain the associated long exact sequence $$ 0 \rightarrow \Hom_G(I^{\otimes n-1},M) \rightarrow \Hom_G(\mathbb{Z}[G/H]\otimes I^{\otimes n-1},M) \xrightarrow{\varkappa} \Hom_G(I^{\otimes n},M) \xrightarrow{\nu} $$ $$ \Ext^1_G(I^{\otimes n-1},M) \xrightarrow{\gamma} \Ext^1_G(\mathbb{Z}[G/H]\otimes I^{\otimes n-1},M) \rightarrow \Ext^1_G(I^{\otimes n},M) \rightarrow \cdots $$ Through the correspondence $f \mapsto \hat{f}$, we get \[H^n([G:H],M) \cong \Hom_G(I^{\otimes n},M)/ \text{Im}(\varkappa).\] We also have by exactness the chain of isomorphisms $$\Hom_G(I^{\otimes n},M)/\text{Im}(\varkappa) \cong \Hom_G(I^{\otimes n},M)/\ker(\nu) \cong \text{Im}(\nu) \cong \ker(\gamma). $$ Moreover, a straightforward generalization of \cite[Lemma 5.4]{FarberMes} for the inclusion $H \hookrightarrow G$ gives us the isomorphism \[\Ext^1_G(\mathbb{Z}[G/H]\otimes I^{\otimes n-1},M) \cong \Ext^1_H(I^{\otimes n-1}, M), \] which is induced by the map \[ \Hom_{G}(\mathbb{Z}[G/H]\otimes I^{\otimes n-1},M) \rightarrow \Hom_{H}(I^{\otimes n-1},M) \] defined by associating to any $G$-homomorphism $f: \mathbb{Z}[G/H]\otimes I^{\otimes n-1} \rightarrow M$ the restriction $f_{|_{e \otimes I^{\otimes n-1}}}$ to $e \otimes I^{\otimes n-1} \subset \mathbb{Z}[G/H]\otimes I^{\otimes n-1}$. Finally, by \cite[Proposition III.2.2]{Brown82}, we have that $$\Ext^1_G(I^{\otimes n-1},M) \cong H^1(G,\Hom_{\mathbb{Z}}(I^{\otimes n-1},M))$$ and $$\Ext^1_H(I^{\otimes n-1},M) \cong H^1(H,\Hom_{\mathbb{Z}}(I^{\otimes n-1}, M)). $$ The action on $\Hom_{\mathbb{Z}}(I^{\otimes n-1},M)$ is defined by $(gf)(x) = g f(g^{-1}x)$ for any $g \in G$, $f \in \Hom_{\mathbb{Z}}(I^{\otimes n-1},M)$ and $x \in I^{\otimes n-1}$. Then, $\gamma$ becomes the restriction homomorphism.
	\end{proof}

	Consider the canonical map 
	$$ \rho\colon EG \rightarrow E_{\fH}G$$ 
	and the chain homotopy homomorphism between the cellular chain complex of $EG$ and the relative standard resolution of $G$ with respect to $H$ which corresponds to sending $g \mapsto gH$, its class in the coset space $G/H$. Now, after applying the functor $\Hom_G(\_,M)$, consider the induced maps on cohomology, which gives a map between the Adamson cohomology and the usual cohomology of the group $$ \rho^*\colon H^*([G:H],M) \rightarrow H^*(G,M). $$ Observe that, at the chain level, this map is induced by the projection $\mZ[G]\to\mZ[G/H]$ and therefore it respects product structures. The following result arises immediately from the definitions involved,  if we take $I$ as the coefficient module in the previous homomorphism. Nonetheless, it is relevant enough to be highlighted on its own:
	
	\begin{proposition}
		With $\rho^*$ defined as before, let $\phi$ and $\omega$ be the Adamson canonical class and the Berstein class relative to $H$, respectively. We have $$ \rho^*(\phi) = \omega. $$  
	\end{proposition}
	
	As a last remark, the naturality of $\rho^*$ with respect to change of coefficient system implies that $\text{Im }\rho^*$ corresponds to the \emph{essential classes} in the sense of \cite{FarberMes}.
	
	\subsection{A spectral sequence}
	
	We will make a brief introduction to the existence of a spectral sequence which contains information about both Adamson and usual cohomology. This sequence is derived from a much more general theory of relative homological algebra developed in \cite{EilenbergMoore65}. We will restrict here to our case of interest.
	
	Take the $(G,H)$-projective resolution of $\mathbb{Z}$  $$ \cdots \rightarrow \mathbb{Z}[G/H] \otimes I^{\otimes n} \rightarrow \cdots \rightarrow \mathbb{Z}[G/H]\otimes I \rightarrow \mathbb{Z}[G/H] \rightarrow \mathbb{Z} \rightarrow 0. $$ Looking at it as an object in the category of sequences of $G$-modules consider a $G$-projective resolution of it, which gives us a double complex
	
	$$ \begin{tikzcd}
		\cdots \arrow[r] & P_{i,j} \arrow[r] \arrow[d] & \cdots \arrow[r] & P_{i,0} \arrow[r] \arrow[d] & Q_i \arrow[r] \arrow[d] & 0 \\
		& \vdots \arrow[d] &  & \vdots \arrow[d]  & \vdots \arrow[d] &  \\
		\cdots \arrow[r] & P_{0,j} \arrow[r] \arrow[d] & \cdots \arrow[r]  & P_{0,0} \arrow[r] \arrow[d] & Q_0 \arrow[r] \arrow[d] & 0 \\
		\cdots \arrow[r] & \mathbb{Z}[G/H] \otimes I^{\otimes j} \arrow[r] & \cdots \arrow[r] & \mathbb{Z}[G/H] \arrow[r] & \mathbb{Z} \arrow[r] & 0 
	\end{tikzcd} $$ such that every $P_{i,j}$ is $G$-projective, every column is a $G$-projective resolution, and each row (except the first one) is split exact. Now, applying the functor $\Hom_G(-,M)$ for some choice of coefficient system $M$, we obtain another double complex and its associated spectral sequence. Let us have a  glance at the horizontal filtration. Given that every row above the first one is split exact we have that $\E_0^{p,q} = \Hom_G(P_{p,q},M)$ and $\E_1^{p,q} = 0$ for $q > 0$. Moreover, $\E_1^{p,0} \cong \Hom_G(Q_p,M)$. As we can see the spectral sequence collapses and, given that $Q_*$ is a projective resolution of $\mathbb{Z}$ as a trivial $G$-module, as stated before, it converges to $\Ext_G^{*}(\mathbb{Z},M) = H^*(G,M)$.
	
	The vertical filtration provides more information. Every column is a projective resolution, so the first page of the spectral sequence has the form $$\E^{p,q}_1  = \Ext_G^{q}(\mathbb{Z}[G/H] \otimes I^{\otimes p},M). $$  The differential on this page is the map $\E_1^{p,q} \xrightarrow{\overline{d_1}} \E_1^{p+1,q}$ induced by the original differential on the $(G,H)$-projective resolution, $$ \mathbb{Z}[G/H] \otimes I^{\otimes p+1} \xrightarrow{ \iota \circ \varepsilon \otimes \id} \mathbb{Z}[G/H] \otimes I^{\otimes p}. $$ Therefore, the second page of the spectral sequence corresponds to $$ \E_2^{p,q} = H^p(\Ext_G^{q}(\mathbb{Z}[G/H]\otimes I^{\otimes *},M)). $$ It is in this second page where, if we restrict to $q = 0$, Adamson cohomology appears. Indeed $ \E^{p,0}_2 = H^p([G:H],M) $ and we have
	
	\begin{proposition}
		There exists a spectral sequence $$\E^{p,q}_2 = H^p(\Ext_G^{q}(\mathbb{Z}[G/H]\otimes I^{\otimes *},M)) \Rightarrow H^{p+q}(G,M) $$ such that $\E^{p,0}_2 = H^p([G:H],M).$
	\end{proposition}
	
	We believe the study of this sequence and, especially, the identification of the second page with a manageable object is worthwhile in itself. The authors hope to obtain more information about it in future work.
	
	\subsection{Adamson vs Bredon cohomology}\label{sect:AdamsonvsBredon}
	
	We will now recast Adamson cohomology in terms of equivariant Bredon cohomology in order to reconcile our approach with that of Farber, Grant, Lupton and Oprea \cite{FGLO17}.
	
	Let us briefly review the definition of Bredon cohomology. Recall that the \textit{orbit category} of $G$ associated to a family $\mathcal{F}$ of subgroups of~$G$, written $\textrm{Or}_{\mathcal{F}}G$, is a category whose object are homogeneous $G$-spaces $G/K$ for $K \in \mathcal{F}$, and morphisms are $G$-equivariant maps between them. A $\textrm{Or}_{\mathcal{F}}G$-\textit{module} is a contravariant functor from $\textrm{Or}_{\mathcal{F}}G$ to the category of abelian groups. A $\textrm{Or}_{\mathcal{F}}G$-\textit{homomorphism} of such modules is a natural transformation. The category of $\textrm{Or}_{\mathcal{F}}G$-modules inherits the structure of an abelian category from the category of abelian groups; in particular, the notion of a projective $\textrm{Or}_{\mathcal{F}}G$-module is defined. If the family contains the trivial subgroup, the \emph{principal component} refers to evaluating the module or morphism in the $G/e$ component.
	
	\begin{example}
		Let $\mathcal{F}$ be a full family of subgroups of $G$. Given a $G$-CW complex~$X$ with isotropy groups in $\mathcal{F}$, define a $\textrm{Or}_{\mathcal{F}}G$-module $\underline{C}_n(X)$ as follows.
		\begin{itemize}
			\item $\underline{C}_n(X)(G/K) = C_n(X^K)$, where $C_n(X^K)$ denotes the group of cellular $n$-chains of $X^K = \{x \in X \,|\, kx = x \textnormal{ for any $k\in K$}\}$.
			\item If $\varphi \colon G/K \to G/L$ is a $G$-equivariant map, then $\varphi(gK) = gg_0L$ for some $g_0 \in G$ such that $g_0^{-1}Kg_0 \subseteq L$. Consequently, $\varphi$ induces a cellular map $X^L \to X^K$, $x \mapsto g_0x$, which descends to the chain level to define a homomorphism $\underline{C}_n(\varphi) \colon C_n(X^L) \to X_n(X^K)$.
		\end{itemize}
		For any $n\geq 1$, there is the obvious $\textrm{Or}_{\mathcal{F}}G$-homomorphism $\underline{d}_n \colon \underline{C}_n(X) \to \underline{C}_{n-1}(X)$, and so we have a $\textrm{Or}_{\mathcal{F}}G$-chain complex $(\underline{C}_*(X), \underline{d}_*)$.
		
	\end{example} 
	
	Using notation from the example above, define the \textit{Bredon cohomology} of $X$ with respect to the family $\mathcal{F}$ and with coefficients in a $\textrm{Or}_{\mathcal{F}}G$-module $\underline{M}$ as 
	$$ H_{\mathcal{F}}^*(X, \underline{M}) = H^*\big(\Hom_{\textrm{Or}_{\mathcal{F}}G}(\underline{C}_*(X), \underline{M})\big). $$
	
	The \textit{Bredon cohomological dimension} of $G$ with respect to $\mathcal{F}$, denoted $\cd_{\mathcal{F}}\,G$, is the length of the shortest possible $\textrm{Or}_{\mathcal{F}}G$-projective resolution of $\underline{\mathbb{Z}}$, where $\underline{\mathbb{Z}}$ is a constant $\textrm{Or}_{\mathcal{F}}G$-module which sends every morphism to $\textrm{id} \colon \mathbb{Z} \to \mathbb{Z}$. Recall that $\langle H \rangle$ denotes the smallest full family of subgroups of $G$ containing $H$.
	
	\begin{theorem}\label{thm:Adamson_as_Bredon}
		Given a $G$-module M, let $\underline{M}$ be the $\textnormal{Or}_{\langle H\rangle} G$-module defined by setting $\underline{M}(G/K) = M^K$. Then 
		$$ H^*\big([G:H],M\big) \cong H^*_{\langle H\rangle}(E_{\langle H \rangle}G, \underline{M}).$$
		In particular, $\cd\,[G:H] \leq \cd_{\langle H \rangle}G$. 
	\end{theorem}
	
	We note that this result has been recently derived with different methods in \cite{Arcin18}, and also in \cite{YalcinPamuk} when $G$ is a finite group.
	
	\begin{proof}
		In what follows, we take as a model for $E_{\langle H \rangle}(G)$ the geometric realization of a suitable $\Delta$-complex such that its cellular chain complex coincides with the standard resolution of $G$ relative to $H$ (for details on the construction see \cite[Proposition 4.16]{Arcin}). For comparing Adamson and Bredon cohomologies, evaluate the cellular $\textrm{Or}_{\langle H \rangle}(G)$-chain complex on the principal component, with gives us $$ C_n^{\langle H \rangle}(E_{\langle H \rangle}(G))(G/e) = \text{H}_n(E_{\langle H \rangle}(G)_{n+1},E_{\langle H \rangle}(G)_n).  $$ By excision, we have that $$  C_n^{\langle H \rangle}(E_{\langle H \rangle}(G))(G/e) = \mathbb{Z}[(G/H)^{n+1}].  $$

		For every $n \geq 0$ define a homomorphism 
		$$ \Phi\colon \Hom_G\big(\mathbb{Z}\big[(G/H)^n\big],M\big) \rightarrow \Hom_{\textrm{Or}_{\langle H \rangle}G}\big(\underline{C}_n(E_{\langle H \rangle}G),\underline{M}\big) $$ 
		by assigning to every $\varphi \in \Hom_G\big(\mathbb{Z}\big[(G/H)^n\big],M\big) $ a map $\varphi_K$ for every subgroup $K \in \langle H \rangle$, defined as the composition
		
		$$ \mathbb{Z}[((G/H)^K)^n] \hookrightarrow \mathbb{Z}[(G/H)^n] \xrightarrow{\varphi} M $$ where the first map is the inclusion (that is, the one induced by the trivial element). If we consider, given $H, K \in \langle H \rangle$, an equivariant map $G/L \rightarrow G/K$, which can be identified as a $g \in G$ such that $gLg^{-1} \leq K$, we have the following diagram
		
		$$
		\begin{tikzcd}
			\mathbb{Z}[((G/H)^K)^n] \arrow[r] \arrow[d]  & \mathbb{Z}[((G/H)^L)^n] \arrow[d] \\
			\mathbb{Z}[(G/H)^n] \arrow[d,"\varphi"'] & \mathbb{Z}[(G/H)^n] \arrow[d, "\varphi"]\\
			M^K \arrow[r] & M^L
		\end{tikzcd}
		$$ where both the top and bottom horizontal morphisms denote action by $g$. Due to the fact that $\varphi$ is a $G$-module homomorphism, the diagram above is commutative. Moreover, $\Phi$ commutes with the differential. Indeed, if we consider $\Phi(\delta \varphi)$ with $\varphi \in \Hom_G(\mathbb{Z}[(G/H)^n],M)$ (and $\delta$ the corresponding differential) we obtain, for every $g\colon G/L \rightarrow G/K$, a diagram analogous to the one above with the top and bottom horizontal arrows being the action by $g$ and the vertical ones the maps which assigns to every tuple $(x_0, \cdots, x_n)$ the element $\sum_i (-1)^i \varphi(x_0, \cdots, \widehat{x_i}, \cdots, x_n)$. Now, considering $\partial$ as the differential in the Bredon complex, for every $g\colon G/L \rightarrow G/K$ $\partial \Phi(\varphi)$ gives us a diagram
		
		$$\small{\begin{tikzcd}
				& \mathbb{Z}[((G/H)^K)^n] \arrow[rr] \arrow[dd, "\varphi_K", near start]  & & \mathbb{Z}[((G/H)^L)^n] \arrow[dd,"\varphi_L"]  \\
				\mathbb{Z}[((G/H)^K)^{n+1}] \arrow[dd, "\varphi^{'}_K"] \arrow[rr, crossing over] \arrow[ru, "d"] & & \mathbb{Z}[((G/H)^L)^{n+1}] \arrow[ru, "d"] \\
				& M^K \arrow[rr] & & M^L\\
				M^K \arrow[rr] & & M^L \arrow[from=uu, "\varphi^{'}_L", near start, crossing over]
		\end{tikzcd}}$$ with the diagonal arrows the respective differentials in the corresponding complexes and $\varphi^{'}_{\ast}$ defined as the composition $\varphi_{\ast} \circ d$. Now, given that every $\varphi_*$ is defined as the composition of an inclusion followed by $\varphi$, $\varphi^{'}_{\ast}$ assigns to every tuple $(x_0, \cdots, x_n)$ the element $\sum_i (-1)^i \varphi(x_0, \cdots, \widehat{x_i}, \cdots, x_n)$. And so $\Phi$ is a well-defined homomorphism of cochain complexes.
		
		Finally, the map $\Phi$ is surjective and injective. In order to see surjectiveness, construct for any map $$\alpha \in \Hom_{\textrm{Or}_{\langle H \rangle}(G)}(C_{\ast}^{\langle H \rangle}(E_{\mathcal{\langle H \rangle}}(G)),\underline{M})$$ and for every $K \in \langle H \rangle$ a diagram $$\begin{tikzcd}
			\mathbb{Z}[((G/H)^K)^n] \arrow[r, hookrightarrow] \arrow[d, "\alpha_K"] & \mathbb{Z}[G/H] \arrow[d, "\alpha_e"]\\
			M^K \arrow[r] & M \end{tikzcd}$$
		where the top vertical arrow is the inclusion induced by $e$. Then such a map $\alpha$ can be seen as the image of $\alpha_e$ via $\Phi$. The injectivity is immediate from the definition of $\Phi$.
		
		Given that $\Phi$ is a bijective map for every $n$, there exists a map \[\Psi \colon \text{Hom}_{\textrm{Or}_{\langle H \rangle}G}\big(\underline{C}_n(E_{\langle H \rangle}G),\underline{M}\big) \rightarrow \text{Hom}_G\big(\mathbb{Z}\big[(G/H)^n\big],M\big) \] such that $\Phi_n \circ \Psi_n$ and $\Psi_n \circ \Phi_n$ are the respective identities for every $n \geq 0$. The map $\Psi$ is easily seen as a chain homomorphism, given that \[ \Psi_{n+1} \circ \partial_n = \Psi_{n+1} \partial_n (\Phi_n \Psi_n) = \Psi_{n+1}(\Phi_{n+1} \delta_n) \Psi_n = \delta_n \circ \Psi_n. \] Finally, we have that $\Psi_n \circ \Phi_n - \text{Id} =  \delta_{n-1} h_n + h_{n+1} \delta_n$ and $\Phi_n \circ \Psi_n - \text{Id} = \partial_{n-1} h'_n + h'_{n+1} \partial_n $ for \[h_n \colon \text{Hom}_G\big(\mathbb{Z}\big[(G/H)^n\big],M\big) \rightarrow \text{Hom}_G\big(\mathbb{Z}\big[(G/H)^{n-1}\big],M\big) \] and \[h'_n \colon \text{Hom}_{\textrm{Or}_{\langle H \rangle}G}\big(\underline{C}_n(E_{\langle H \rangle}G),\underline{M}\big) \rightarrow \text{Hom}_{\textrm{Or}_{\langle H \rangle}G}\big(\underline{C}_{n-1}(E_{\langle H \rangle}G),\underline{M}\big) \] diagonal maps corresponding with sending every element in their respective domains to $0$. Thus $\Phi$ defines a chain homotopy equivalence between the Adamson and Bredon cochain complexes, which gives us the desired isomorphism $ H^*\big([G:H],M\big) \cong H^*_{\langle H\rangle}(E_{\langle H \rangle}G, \underline{M}).$
	\end{proof}
	
	Even though Bredon cohomology theory has raised an extensive amount of research since its very inception, the main setback is still the high difficulty of making not only explicit computations, but also of obtaining good bounds for cohomological dimension in most cases. In the face of this structural difficulty Adamson cohomology offers a simpler tool, both theoretically and computationally that, as we just showed, allows to bound Bredon cohomological dimension from below. The natural question that arises is when does Adamson cohomological dimension detect Bredon cohomological dimension?
	
	The most natural example of coincidence of both dimensions happens, as expected, when the subgroup is normal. Indeed, consider $H \leq G$ a subgroup and $K \triangleleft G$ is a normal subgroup contained in $H$. The group $G$ acts on $E_{\langle H/K \rangle}(G/K)$ through the natural projection to the lateral classes by $K$, and if we take an $H/K$-fixed point $p \in E_{\langle H/K \rangle}(G/K)$, $p$ is also an $H$-fixed point, given that $h p = hK p = p$ for any $h \in H$. Since the projection onto the lateral classes by $K$ sends $ghg^{-1}$ to $gK(hK)g^{-1}K$, the definition of the family $\langle H \rangle$ and the universal property of $E_{\langle H \rangle}(G)$ gives us a way of relating the model of the classifying space of $G$ with respect to $\langle H \rangle$ and the classifying space modulo $K$:    
	
	\begin{proposition}[\cite{Arcin}, Proposition 4.21 and Corollary 4.22]\label{UniversalSpnormalsubg}
		Let $H \leq G$, and $K \triangleleft G$ a normal subgroup of $G$ contained in H. Then, a model for $E_{\langle H/K \rangle}(G/K)$ is also a model for $E_{\langle H \rangle}(G)$. In particular, if $H$ is normal in $G$, $E(G/H)$ is a model for $E_{\langle H \rangle}(G)$. 
	\end{proposition}
	
	The natural future line of work here is to investigate in which other cases Adamson cohomological dimension is enough to detect Bredon cohomological dimension, and to study how to control and bound the differences between them when they differ. 
	
	Additionally, despite the fact that Adamson cohomology is easier to approach than Bredon cohomology, in general it is not a simple task to make explicit calculations. As such, there is ample room for future investigation of ways of computing Adamson cohomological dimension. In particular, the authors trust the naturality of the Adamson canonical class will prove fruithful in this matter.

	\section{Final remarks on $\secat(H\hookrightarrow G)$}

	In view of \cite[Corollary 3.5.1]{FGLO17}, $\tc(\pi) \leq \textrm{cd}_{\langle\Delta_{\pi}\rangle}(\pi \times \pi)$ under certain mild assumptions on $\pi$. By its generalization in \ref{secatlessdim} and the definition of Bredon cohomological dimension, we know that $\secat(H \hookrightarrow G) \leq \cd_{\langle H \rangle}G$. It is therefore hard not to ask whether $\tc(\pi) = \cd_{\langle \Delta_{\pi}\rangle}(\pi \times \pi)$ or, more generally, whether $\secat(H\hookrightarrow G) = \cd_{\langle H\rangle} G$. The latter cannot possibly be true, as the following examples show.
	
	\begin{example}
		(1) Consider the inclusion $2 \mathbb{Z} \hookrightarrow \mathbb{Z}$. By \ref{propertysecat}, $\secat(2\mathbb{Z} \hookrightarrow \mathbb{Z}) = 1$. On the other hand, given that the subgroup is normal, the Adamson cohomology coincides with the usual cohomology of the quotient (see \cite[Theorem 3.2]{Adamson}) and then $H^*([\mathbb{Z}: 2\mathbb{Z}]) = H^*(\mathbb{Z}/2 \mathbb{Z})$. Therefore $\cd[G:H] = \cd\,\mathbb{Z}_2 = \cd_{\langle 2\mathbb{Z}\rangle} \mathbb{Z} = \infty$.
		
		(2) It is perhaps interesting to note that this phenomenon is not torsion-related. Consider the inclusion $[F_n, F_n] \hookrightarrow F_n$, where $F_n$ denotes the free group of $n$ generators, and $[F_n,F_n]$ its commutator subgroup. Similarly as above, $\secat\big([F_n, F_n] \hookrightarrow F_n\big) = 1$, but $\cd[F_n : [F_n, F_n]] = \cd_{\langle [F_n, F_n] \rangle} F_n = \cd\,\mathbb{Z}^n = n$.
	\end{example}
	
	Nevertheless, it is possible to find cases where sectional category and Adamson cohomological dimension coincide, as in the next example. 
	
	\begin{example}
		(1) Recall that a group $G$ is said to be nilpotent of order $n$ if there exists a series of normal subgroups $$ \{ 1 \} = G_0 \triangleleft G_1 \triangleleft \cdots \triangleleft G_n = G $$ where $G_{i+1}/G_i \leq Z(G/G_i)$ (equivalently $[G,G_{i+1}] \leq G_i$). Consider the group $$ H_3 = \langle a_1, a_2, b \vert [a_1,b] = 1, [a_2,b] = 1, [a_2,a_1] = b \rangle  $$ known as the \textit{three dimensional Heisenberg group}. This is one of the most paradigmatic torsion-free nilpotent groups. The infinite cyclic group $N$ generated by $b$ is a central subgroup of $H_3$ and $H_3$ fits in a central group extension $$ 0 \rightarrow N \rightarrow H_3 \rightarrow F \rightarrow 0 $$ where $F$ is a free abelian group with basis in one-to-one correspondence with the generators of $H_3$. Given that $N$ is central (and so is normal) in $H_3$, by \cite[Theorem 3.2]{Adamson} we have $H^*([H_3:N]) = H^*(F)$, with $H^n(F) \cong \mathbb{Z}^{\binom{2}{n}}$. The details on how to obtain the cohomology ring structure for integer coefficients of $H_3$ can be consulted in \cite{Huebschmann89}. It can be seen that the nilpotence of the kernel of the homomorphism in cohomology induced by the inclusion $N \hookrightarrow H_3$ is $2$ and so, by Theorem \ref{propertysecat}, we obtain $2 \leq \secat(N \hookrightarrow H_3) \leq 3 = \cd(H_3)$. Proposition \ref{UniversalSpnormalsubg} gives us that a model for $E_{\langle N \rangle}H_3$ is homotopically equivalent to $EH_3/N$, and then $\dim E_{\langle N \rangle}H_3 = \cd[H_3:N] = \cd(H_3/N) = \cd(F) = 2 $. By Theorem \ref{thm:secat_FGLO_char} and its corollary \ref{secatlessdim} we know that $\secat(N \hookrightarrow H_3) \leq \dim E_{\langle N \rangle}H_3$. Therefore $\secat(N \hookrightarrow H_3) \leq \cd[H_3:N] = 2$ and thus $$\secat(N \hookrightarrow H_3) = \cd[H_3:N] = 2.$$  
		
		(2) Consider the subgroup inclusion of $$\mathbb{Z} \hookrightarrow \mathbb{Z} \times \mathbb{Z} \rightarrow \mathbb{Z}$$ taken as the inclusion in the first factor. This can be represented by the fibration \[ S^1 \times \mathbb{R} \xrightarrow{\id \times \text{exp}} T^2.  \] where $\text{exp}$ denotes the exponential map $\text{exp} \colon \mathbb{R} \rightarrow S^1$ defined by $\text{exp}(\theta) = e^{i \theta}$. Looking at $T^2$ as $S^1 \times S^1$, take as an open cover the one defined by $U_0 = S^1 \times S^1\setminus\{1\} $ and $U_1 = S^1 \times S^1 \setminus\{ -1 \}$ (or, equivalently, any choice of two antipodal points). For each $U_i$, we have a local section of the fibration, defined by taking the identity on the $S^1$ factor, and lifting the other factor to $\mathbb{R}$ by considering the argument of the exponential map, i.e. $s_i(e^{i \theta_1}, e^{i \theta_2}) = (e^{i \theta_1}, \theta_2)$. This informs us that $\secat(\mathbb{Z} \hookrightarrow \mathbb{Z} \times \mathbb{Z}) = 1$. By the normality of the subgroup, we have that $\cd[\mathbb{Z} \times \mathbb{Z}: \mathbb{Z}] = \cd(\mathbb{Z}) = 1$ and consequently $\secat(\mathbb{Z} \hookrightarrow \mathbb{Z} \times \mathbb{Z}) = \cd[\mathbb{Z} \times \mathbb{Z}: \mathbb{Z}] $  
		
	\end{example}
	
	As an open question, it remains to elucidate the full relationship between sectional category of subgroup inclusions and Adamson cohomological dimension with respect to the subgroup, and in which cases the former can be represented, or bounded in some direction, by the latter. This line of work, of course, is strongly related to the investigation of how good Adamson cohomological dimension detects Bredon cohomological dimension, and how big the difference can be in interesting cases, as discussed at the end of section 3. It is also important to note that these cases do not provide much information in the context of topological complexity, given that the diagonal subgroup is normal in the product group only under the assumption that the group is abelian. As such, investigating the full scope of the ideas developed in this paper for the case of the diagonal subgroup inclusion should be a matter for future work. 
	
	In \cite{FGLO17} an analogue of the Costa--Farber canonical class is defined in the context of Bredon cohomology, $\textbf{u} \in H_{\langle\Delta_{\pi}\rangle}^1(\pi \times \pi, \underline{I})$. This class is universal. Moreover, the image of this class via the homomorphism $\rho^*$ is precisely the usual Costa-Farber class. If instead of the family $\langle\Delta_{\pi}\rangle$ generated by the diagonal subgroup, we take the family of subgroups $\mathcal{F}$, generated by a subgroup $H \leq G$, it is possible to define a cohomology class represented by the short exact sequence of Bredon modules $$ 0 \rightarrow \underline{I} \rightarrow \mathbb{Z}[?,G/H] \xrightarrow{\varepsilon} \underline{\mathbb{Z}} \rightarrow 0 $$ where $\underline{I}$ is the kernel of the augmentation $\varepsilon$. Let us denote it also, abusing notation, by $\textbf{u}$. This class is the canonical class associated to the family $\mathcal{F}$, and arguments analogues to the case of the diagonal family show that it is also universal, and, evaluating in the principal component, it is immediate that its image by the principal component evaluation homomorphism $\rho^1\colon H^1_{\langle H \rangle}(G,\underline{I}) \to H^1(G,\underline{I}(G/e))$ is the Berstein class of $G$ relative to $H$ introduced in Subsection \ref{sect:Bernstein_class}.
	
	\begin{remark}
		To our knowledge, despite Theorem \ref{thm:Adamson_as_Bredon}, universality of the Bredon class does not imply in a straightforward manner universality of the Adamson class. This is due to the fact that $I^H$ need not coincide with $\underline{I}(G/H)$.
	\end{remark}

	Write $\rho_{\langle H \rangle}$ for the greatest integer $n \geq 0$ such that the principal component evaluation homomorphism
	\[ \rho^n \colon H^n_{\langle H \rangle}(G,\underline{M}) \to H^n(G,\underline{M}(G/e)) \]
	is non-trivial for some $\textrm{Or}_{\langle H \rangle}G$-module $\underline{M}$. A straightforward generalization of \cite[Theorem 4.1]{FGLO17}, using Theorem \ref{thm:secat_FGLO_char}, shows that \[\rho_{\langle H\rangle}\le \secat(H\hookrightarrow G).\]
	The next result shows that this lower bound for sectional category is never better than the standard cohomological lower bound.
	
	\begin{proposition}\label{prop:heightGtRho}
		With the notation above, $\height(\omega) \ge \rho_{\langle H \rangle} $.
	\end{proposition}
	
	\begin{proof}
		Suppose there exists $\alpha \in H^n_{\langle H \rangle}(G,\underline{M})$ such that $\rho^*(\alpha) \neq 0$. Universality of $\textbf{u}$ implies that there exists an $\textrm{Or}_{\langle F\rangle}G$-homomorphism $f \colon \underline{I}^{n} \to \underline{M}$ such that $f^*(\textbf{u}^n) = \alpha$. But then $f$ induces also a $G$-module homomorphism between the principal components of Bredon modules, and thus, it gives a commutative diagram of group cohomologies $$ \begin{tikzcd}
			H_{\langle H \rangle}^n(G, \underline{I}^{n}) \arrow[r, "\rho^*"] \arrow[d, "f^*"]  & H^n(G, I^{\otimes n}) \arrow[d, "f^*"]  \\
			H_{\langle H \rangle}^n(G, \underline{M}) \arrow[r, "\rho^*"']  & H^n(G, M).
		\end{tikzcd} $$ By hypothesis $\rho^*(\alpha) \neq 0$ so $\rho^*(f^*(\textbf{u}^n)) \neq 0$, and, by commutativity, $\rho^*(\textbf{u}^n) = \omega^n$, the Berstein class of $G$ relative to $H$, is nonzero.
	\end{proof}
	
	The last proposition allows for a particularly simple proof of Theorem \ref{th:GenerCostaFarber}.
	
	\begin{proof}[Proof of Theorem \ref{th:GenerCostaFarber}] 
		The ``only if'' part is an immediate consequence of the kernel-nilpotency lower bound for sectional category, see \ref{propertysecat}. For the converse statement, recall that the extension problem
		$$ \begin{tikzcd}
			EG_{n-1} \arrow[rr, "\rho"] \arrow[rd, hook] & & (E_{\fH}G)_{n-1} \\
			& EG_{n} \arrow[ru, dashrightarrow] 
		\end{tikzcd} $$  
		has a solution provided that the cocycle $c^{n}(\rho)$ representing the extension is cohomologous to zero in $H^{n}\big(G, \pi_{n-1}(E_{\fH}G)_{n-1}\big)$. Let us take a closer look at how the obstruction cocycle arises; further details can be found in \cite[Chapter II.3]{TomDieck}. Write $[\rho]$ for the $G$-homotopy class of $\rho \colon EG_{n-1} \to (E_{\fH}G)_{n-1}$. Note that both $EG_{n-1}$ and $(E_{\fH}G)_{n-1}$ are $(n-2)$-connected spaces, and the pair $(EG_n,EG_{n-1})$ is \mbox{$(n-1)$}-connected, hence the (relative) Hurewicz homomorphism gives isomorphisms
		\begin{align*}
			\pi_n(EG_n, EG_{n-1}) &\to H_n(EG_n, EG_{n-1}),\\ 
			\pi_{n-1}(EG_{n-1}) &\to H_{n-1}(EG_{n-1}),\\
			\pi_{n-1}\big((E_{\fH}G)_{n-1}\big) &\to H_{n-1}\big((E_{\fH}G)_{n-1}\big).
		\end{align*}
		Consequently, we have a diagram 
		$$\begin{tikzcd}
			\pi_n(EG_n, EG_{n-1}) \arrow[r, "\partial"] \arrow[d, "\varrho"] & \pi_{n-1}(EG_{n-1}) \arrow[r, "\rho_*"] \arrow[d,"\varrho"] & \pi_{n-1}\big((E_{\fH}G)_{n-1}\big) \arrow[d,"\varrho"] \\
			H_n(EG_n, EG_{n-1}) & H_{n-1}(EG_{n-1}) & H_{n-1}\big((E_{\fH}G)_{n-1}\big),
		\end{tikzcd}$$ 
		where $\partial$ is the boundary operator of the long exact sequence of homotopy groups of the pair $(EG_n, EG_{n-1})$ and can be identified as the corresponding epimorphism over the kernel of the $n$-differential in the cellular chain complex via the Hurewicz isomorphisms. The obstruction cocycle associated to $\rho$ is defined as
		$$ c^{n}(\rho) = \rho_* \partial \varrho^{-1}. $$
		
		By hypothesis and the previous proposition, we conclude $\rho^*$ is trivial in degree~$n$. Naturality of the Hurewicz isomorphism implies that the obstruction class lives in the image of $\rho^n$ and therefore it must be zero. The Eilenberg-Ganea theorem (\cite[Theorem 1]{EG65}, also \cite[Theorem 7.1]{Brown82}) states that $\cd(G) = n$ implies the existence of a $n$-dimensional $EG$ thus, by dimensional reasons, the map $\rho$ can be then extended to the whole space $EG$. Consequently, $\secat(H \hookrightarrow G) \leq n-1$ by Theorem \ref{thm:secat_FGLO_char}.

	\end{proof}
	
	\noindent\textbf{Acknowledgements.}  The first author has been supported by the National Science Centre grant 2015/19/B/ST1/01458. The second and third authors have been supported by the National Science Centre grant 2016/21/P/ST1/03460 within the European Union's Horizon 2020 research and innovation programme under the Marie Sk\l{}odowska-Curie grant agreement No. 665778.
	
	The authors want to thank the referee for his/her thorough review, and for the many useful commentaries provided that have helped to improve the quality of this paper. 
	
	\begin{flushright}
		\includegraphics[width=38px]{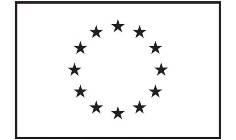}%
	\end{flushright}
	
	\bibliography{bibliography}{}
	\bibliographystyle{plain}
	
\end{document}